\PassOptionsToPackage{unicode}{hyperref}
\PassOptionsToPackage{hyphens}{url}
\documentclass[
]{article}
\title{Matrix decompositions over the double numbers}
\author{Ran Gutin \\ Department of Computer Science, Imperial College London}
\date{\today}

\usepackage{amsmath,amssymb}
\usepackage{lmodern}
\usepackage{iftex}
\ifPDFTeX
  \usepackage[T1]{fontenc}
  \usepackage[utf8]{inputenc}
  \usepackage{textcomp} 
\else 
  \usepackage{unicode-math}
  \defaultfontfeatures{Scale=MatchLowercase}
  \defaultfontfeatures[\rmfamily]{Ligatures=TeX,Scale=1}
\fi
\IfFileExists{upquote.sty}{\usepackage{upquote}}{}
\IfFileExists{microtype.sty}{
  \usepackage[]{microtype}
  \UseMicrotypeSet[protrusion]{basicmath} 
}{}
\makeatletter
\@ifundefined{KOMAClassName}{
  \IfFileExists{parskip.sty}{%
    \usepackage{parskip}
  }{
    \setlength{\parindent}{0pt}
    \setlength{\parskip}{6pt plus 2pt minus 1pt}}
}{
  \KOMAoptions{parskip=half}}
\makeatother
\usepackage{xcolor}
\IfFileExists{xurl.sty}{\usepackage{xurl}}{} 
\IfFileExists{bookmark.sty}{\usepackage{bookmark}}{\usepackage{hyperref}}
\hypersetup{
  pdftitle={Matrix decompositions over the double numbers},
  pdfauthor={Ran Gutin (Department of Computer Science, Imperial College London)},
  hidelinks,
  pdfcreator={LaTeX via pandoc}}
\usepackage{color}
\usepackage{fancyvrb}

\DefineVerbatimEnvironment{Highlighting}{Verbatim}{commandchars=\\\{\}}
\newenvironment{Shaded}{}{}

\newcommand{\BuiltInTok}[1]{#1}

\newcommand{\CommentTok}[1]{\textcolor[rgb]{0.38,0.63,0.69}{\textit{#1}}}

\newcommand{\ControlFlowTok}[1]{\textcolor[rgb]{0.00,0.44,0.13}{\textbf{#1}}}

\newcommand{\DecValTok}[1]{\textcolor[rgb]{0.25,0.63,0.44}{#1}}

\newcommand{\ImportTok}[1]{#1}

\newcommand{\KeywordTok}[1]{\textcolor[rgb]{0.00,0.44,0.13}{\textbf{#1}}}
\newcommand{\NormalTok}[1]{#1}
\newcommand{\OperatorTok}[1]{\textcolor[rgb]{0.40,0.40,0.40}{#1}}

\newcommand{\StringTok}[1]{\textcolor[rgb]{0.25,0.44,0.63}{#1}}

\providecommand{\tightlist}{%
  \setlength{\itemsep}{0pt}\setlength{\parskip}{0pt}}
\usepackage{amsthm} \newtheorem{definition}{Definition}[section] \newtheorem{lemma}{Lemma}[section] \newtheorem{theorem}{Theorem}[section]   
\ifLuaTeX
  \usepackage{selnolig}  
\fi

\begin{document}
\maketitle
\begin{abstract}
  We take matrix decompositions that are usually applied to matrices over
  the real numbers or complex numbers, and extend them to matrices over an
  algebra called the double numbers. In doing so, we unify some
  matrix decompositions: For instance, we reduce the LU decomposition of
  real matrices to LDL decomposition of double matrices. We
  similarly reduce eigendecomposition of real matrices to singular value
  decomposition of double matrices. Notably, these are opposite to
  the usual reductions. This provides insight into linear algebra over the
  familiar real numbers and complex numbers. We also show that algorithms
  that are valid for complex matrices are often equally valid for
  double matrices. We finish by proposing a new matrix
  decomposition called the Jordan SVD, which we use to challenge a claim
  made in Yaglom's book Complex Numbers In Geometry concerning Linear
  Fractional Transformations over the double numbers.
\end{abstract}

\hypertarget{introduction}{%
\section{Introduction}\label{introduction}}

In this paper, we consider extending numerical linear algebra to a
hypercomplex number system called the double numbers (and related
number systems like the double-complex numbers / tessarines). Specifically,
we are interested in the subject of matrix decompositions.

The process we follow is as follows:

\begin{enumerate}
\def\labelenumi{\arabic{enumi}.}
\tightlist
\item
  Take a matrix decomposition over the complex numbers (such as LDL,
  SVD, QR).
\item
  Extend it to double matrices via analogy.
\item
  Break it into real components.
\item
  Observe what decompositions of real matrices emerge out.
\end{enumerate}

We often find that the outcome of step 4 is interesting, in that the
decomposition of real matrices that emerges from step 4 is often
something recognisable. Some examples include: If in step 1, we start
with the LDL decomposition (which only applies to Hermitian matrices),
then in step 4 we derive the LU decomposition (which applies to most
other matrices). Another example is that if we start with the Singular
Value Decomposition, which is often seen as a special case of
eigendecomposition, then in step 4 we derive precisely
eigendecomposition.

We also observe that algorithms that are valid for the decomposition in
step 1 can be extended to the decomposition in step 4. Sometimes this
extension of algorithms can be done automatically.

The third step (breaking into real components) leads us to a minor
novelty: We write a double matrix using the notation \([A,B]\),
which we define to mean \(A\frac{1 + j}2 + B^T\frac{1 - j}2\) where
\(A\) and \(B\) are real matrices. Crucially, note that \(B\) is
transposed. The symbol \(j\) is defined below.

\hypertarget{definitions}{%
\section{Definitions}\label{definitions}}

\hypertarget{sec:definitions}{%
\subsection{Double numbers and double
matrices}\label{sec:definitions}}

The \emph{double numbers}, called split-complex numbers on Wikipedia, and \emph{perplex numbers} in some other sources) are the hypercomplex number system in which every number is of the form

\[a + bj, a \in \mathbb R, b \in \mathbb R\]

where \(j\) is a number that satisfies \(j^2 = 1\) without \(j\) being
either \(1\) or \(-1\).

The definition of the arithmetic operations should be clear from the
above description. Define \((a + bj)^* = a - bj\) (similar to complex
conjugation).

Consider the basis \(e = \frac{1+j}2\) and \(e^* = \frac{1-j}2\).
Observe that \(e^2 = e\), \((e^*)^2 = e^*\) and \(e e^* = e^* e = 0\).
Thus \((a e + b e^*) (c e + d e^*) = (ac) e + (bd) e^*\). In other
words, multiplication of double numbers is componentwise in this
basis. Thanks to this, we see that the double numbers can be
defined as the algebra \(\mathbb R \oplus \mathbb R\), where \(\oplus\)
denotes the direct sum of algebras. Observe that
\((a e + b e^*)^* = b e + a e^*\).

We shall consider matrices of double numbers. A matrix \(M\) of
double numbers can be written in the form
\(A \frac{1 + j}2 + B^T \frac{1 - j}2\) where \(A\) and \(B\) are real
matrices. Write \([A,B]\) for \(A \frac{1 + j}2 + B^T \frac{1 - j}2\).
Observe the following identities: \[\begin{aligned}
\\
[A,B] + [C,D] &= [A+C,B+D] \\
[A,B] \times [C,D] &= [AC, DB] \\
[A,B]^* &= [B,A]
\end{aligned}\]

In the above \([A,B]^*\) refers to the conjugate-transpose operation.
Notice that the second component of \([A,B] \times [C,D]\) is \(DB\),
not \(BD\).

\hypertarget{sec:families}{%
\subsection{Families of double matrices}\label{sec:families}}

We say that a double matrix \(M\) is \emph{Hermitian} if it satisfies \(M^* = M\). Observe then that a Hermitian matrix is precisely of the form \([A,A]\).

We say that a double matrix \(M\) is \emph{unitary} if it satisfies \(M^* M = I\). Observe then that a unitary matrix is precisely of the form \([A,A^{-1}]\).

We say that a double matrix \(M\) is \emph{lower triangular} if it is of the form \([L,U]\) where \(L\) is a lower triangular real matrix and \(U\) is an upper triangular real matrix. Similarly, a matrix of the form \([U,L]\) is called \emph{upper triangular}.

A double matrix is called \emph{diagonal} if it is of the form \([D,E]\) where \(D\) and \(E\) are diagonal real matrices.

A double matrix is real precisely when it is of the form \([A,A^T]\).

\hypertarget{double-complex-numbers-and-matrices}{%
\subsection{Double-complex numbers and matrices}\label{double-complex-numbers-and-matrices}}

A \emph{double-complex number} (\cite{doublecomplex1}) is a number of the form \(w + zj\) where \(w\) and \(z\) are complex numbers, and \(j^2 = +1\). Multiplication is the same as over the double numbers, namely: \((w + zj)(w' + z'j) = ww' + zz' + j(wz' + zw')\). The product is a commutative one. The definition of conjugation is \((w + zj)^* = w - zj\). Sometimes these numbers are called \emph{tessarines}.

By a \emph{double-complex matrix}, we mean a matrix whose entries are
double-complex numbers. A double-complex matrix can always be expressed in the
form \([A,B]\), which we define to mean
\(A \frac{1+j}{2} + B^T \frac{1-j}{2}\) where \(A\) and \(B\) are
complex matrices. Observe that we use the tranpose of \(B\) rather than
the conjugate-transpose of \(B\).

Observe then the following identities: \[\begin{aligned}
\\
[A,B] + [C,D] &= [A+C,B+D] \\
[A,B] \times [C,D] &= [AC, DB] \\
[A,B]^* &= [B,A]
\end{aligned}\]

The various families of matrices (Hermitian, unitary, triangular,
diagonal) have identical descriptions over the double-complex numbers as they
do over the double numbers. A double-complex matrix \(M\) is a complex
matrix precisely when it is of the form \(M = [A,A^T]\) where \(A\) is a
complex matrix.

We need these numbers in order to study Singular Value Decomposition
over double or double-complex matrices.

\hypertarget{computing-with-double-numbers}{%
\subsection{Computing with double numbers}\label{computing-with-double-numbers}}

In order to present the algorithms in this paper, we use the Python library Sympy (version 1.6, \cite{sympy}). We have also implemented some of the algorithms below using the C++ library Eigen (\cite{eigen}). We have not included the Eigen code here. Eigen is significantly (likely many orders of magnitude) faster than Sympy for the purposes of computing with double matrices. The slow speed of Sympy is caused by the fact that it's a symbolic computation library, and when using it, we represent a double number as a symbolic expression in the form \(a + b\cdot j\). We need to implement the following functions to be able to do experiments with double numbers:

\begin{Shaded}
\begin{Highlighting}[]
\ImportTok{from}\NormalTok{ sympy }\ImportTok{import} \OperatorTok{*}

\NormalTok{J }\OperatorTok{=}\NormalTok{ var(}\StringTok{\textquotesingle{}J\textquotesingle{}}\NormalTok{)}

\KeywordTok{def}\NormalTok{ simplify\_split(e):}
\NormalTok{    real\_part }\OperatorTok{=}\NormalTok{ (e.subs(J,}\DecValTok{1}\NormalTok{) }\OperatorTok{+}\NormalTok{ e.subs(J,}\OperatorTok{{-}}\DecValTok{1}\NormalTok{))}\OperatorTok{/}\DecValTok{2}
\NormalTok{    imag\_part }\OperatorTok{=}\NormalTok{ (e.subs(J,}\DecValTok{1}\NormalTok{) }\OperatorTok{{-}}\NormalTok{ e.subs(J,}\OperatorTok{{-}}\DecValTok{1}\NormalTok{))}\OperatorTok{/}\DecValTok{2}
    \ControlFlowTok{return}\NormalTok{ simplify(real\_part }\OperatorTok{+}\NormalTok{ J}\OperatorTok{*}\NormalTok{imag\_part)    }

\KeywordTok{def}\NormalTok{ first\_part(m):}
    \ControlFlowTok{return}\NormalTok{ simplify\_split((}\DecValTok{1} \OperatorTok{+}\NormalTok{ J) }\OperatorTok{*}\NormalTok{ m).subs(J,}\DecValTok{0}\NormalTok{)}

\KeywordTok{def}\NormalTok{ second\_part(m):}
    \ControlFlowTok{return}\NormalTok{ simplify\_split((}\DecValTok{1} \OperatorTok{{-}}\NormalTok{ J) }\OperatorTok{*}\NormalTok{ m).subs(J,}\DecValTok{0}\NormalTok{)}

\KeywordTok{def}\NormalTok{ scabs(x):}
    \CommentTok{"""Absolute value for split{-}complex numbers."""}
    \ControlFlowTok{return}\NormalTok{ sqrt(}\BuiltInTok{abs}\NormalTok{(first\_part(x)) }\OperatorTok{*} \BuiltInTok{abs}\NormalTok{(second\_part(x)))}
\end{Highlighting}
\end{Shaded}

The function \texttt{simplify\_split} is particularly useful for simplifying expressions involving double numbers. It is based on the identity \(f{\left(a + b j \right)} = \frac{f{\left(a + b \right)} + f{\left(a - b \right)}}{2} + j \frac{f{\left(a + b \right)} - f{\left(a - b \right)}}{2}\) where \(f\) is any analytic function over \(\mathbb R\). We end up using it a lot in our algorithms, though it's completely unnecessary if one doesn't use symbolic algebra the way we have. The functions \texttt{first\_part} and \texttt{second\_part} are useful for extracting multiples of the basis vectors \(e = \frac{1 + j}2\) and \(e^* = \frac{1 - j}2\) respectively. To extract \(B\) from \(M = [A,B]\), one needs to compute \texttt{second\_part(M.T)} (where \texttt{M.T} means the transpose of \(M\)). The expression \(M^*\) can be expressed as \texttt{M.T.subs(J,-J)}.

It is straightforward to adapt the code to other Computer Algebra Systems like Sagemath (\cite{sagemath}).

\hypertarget{simple-matrix-decompositions-over-the-double-numbers}{%
\section{Simple matrix decompositions over the double numbers}\label{simple-matrix-decompositions-over-the-double-numbers}}

In this section, we will consider analogues of various matrix decompositions over the double numbers. We will observe by taking components (in other words, by extracting \(A\) and \(B\) from \([A,B]\)) that these decompositions are equivalent to well-known decompositions of real matrices.

\hypertarget{ldl-decomposition-a-variant-of-cholesky}{%
\subsection{LDL decomposition (a variant of Cholesky)}\label{ldl-decomposition-a-variant-of-cholesky}}

Recall the LDL decomposition (\cite{cholesky,computations}): Let \(M\) be a Hermitian complex matrix; we have that \(M\) can almost always be expressed in the form \[M = LDL^*\] where \(L\) is a lower triangular matrix and \(D\) is a diagonal Hermitian matrix. What is the analogue of this over the double numbers? If we interpret the above decomposition over the double numbers using the definitions in section \ref{sec:families}, we get \[[A,A] = [L,U] [D,D] [L,U]^*,\] or in other words \[A = L D U,\] which is the familiar LDU decomposition.

Notice that while it's known that LDL trivially reduces to an instance of LU (a more general decomposition), it is not as obvious that a reduction in the opposite direction is possible. We have demonstrated that such a thing is possible.

\hypertarget{singular-value-decomposition}{%
\subsection{Singular Value Decomposition}\label{singular-value-decomposition}}

In the following, we work with square matrices. It is possible to generalise to non-square matrices.

Consider the singular value decomposition (\cite{strang09}): \[M = U S V^*\] where \(U\) and \(V\) are unitary, and \(S\) is diagonal and Hermitian. If we interpret this over the double numbers (using section \ref{sec:families}), we get \[[A,B] = [P,P^{-1}] [D,D] [Q^{-1},Q]\] We break into components and get \[A = PDQ^{-1}, B = QDP^{-1}.\] Finally, observe that \(AB = P D^2 P^{-1}\) and \(BA = Q D^2 Q^{-1}\). In other words, we get the familiar eigendecompositions of \(AB\) and \(BA\) (which it turns out have the same eigenvalues).

Take note though that in order for a double matrix \([A,B]\) to have an SVD, it may be necessary to work with a \emph{complexification} of the double numbers. This complexification is sometimes called the \emph{double-complex numbers} or the \emph{tessarines}.

Finally, notice that while it's obvious that SVD is reducible to eigendecomposition (namely, the eigendecompositions of \(A^* A\) and \(A A^*\)) it is less obvious that a reduction in the opposite direction is possible. We have provided such a reduction here.

\hypertarget{qr-decomposition}{%
\subsection{QR decomposition}\label{qr-decomposition}}

Recall that the QR decomposition (\cite{computations}) of a complex matrix \(M\) is of the form \[M = QR\] where \(Q\) is a unitary matrix and \(R\) is an upper triangular matrix.

Extend the above to double matrices to arrive at \[[A,B] = [C,C^{-1}] [U,L].\] We get \[A = CU, B = LC^{-1},\] and therefore that \(BA = LU\). In other words, the analogue of QR decomposition over double matrices yields the LU decomposition of the product of two real matrices.

\hypertarget{algorithms}{%
\section{Algorithms}\label{algorithms}}

In this section, we demonstrate how the standard algorithms for LDL decomposition, QR decomposition, and Singular Value Decomposition generalise to double matrices.

\hypertarget{ldl-decomposition}{%
\subsection{LDL decomposition}\label{ldl-decomposition}}

The following relations (\cite{computations}) can be used to compute the LDL decomposition of a complex matrix:

\[\begin{aligned}
D_j &= A_{jj} - \sum_{k=1}^{j-1} L_{jk}L_{jk}^* D_k, \\
L_{ij} &= \frac{1}{D_j} \left( A_{ij} - \sum_{k=1}^{j-1} L_{ik} L_{jk}^* D_k \right) \quad \text{for } i>j
\end{aligned}\]

These relations generalise to double matrices, and enable one to compute the LDL decomposition of a double matrix in the same way. In doing so, we derive an algorithm (equivalent to the standard one) for computing the LDU decomposition of a real matrix.

\hypertarget{singular-value-decomposition-1}{%
\subsection{Singular Value Decomposition}\label{singular-value-decomposition-1}}

There are multiple algorithms for computing the singular values of a matrix, which can usually be adapted for computing the singular vectors as well. Here, for simplicity's sake, we will focus on only computing the singular values. The simplest such algorithm is the following (\cite{implicitcholesky}):

\begin{Shaded}
\begin{Highlighting}[]
\KeywordTok{def}\NormalTok{ svd(A, iters}\OperatorTok{=}\DecValTok{20}\NormalTok{):}
\NormalTok{    M }\OperatorTok{=}\NormalTok{ A.H }\OperatorTok{*}\NormalTok{ A}
    \ControlFlowTok{for}\NormalTok{ \_ }\KeywordTok{in} \BuiltInTok{range}\NormalTok{(iters):}
\NormalTok{        L, D }\OperatorTok{=}\NormalTok{ Matrix.LDLdecomposition(M)}
\NormalTok{        M }\OperatorTok{=}\NormalTok{ D }\OperatorTok{**}\NormalTok{ Rational(}\DecValTok{1}\NormalTok{,}\DecValTok{2}\NormalTok{) }\OperatorTok{*}\NormalTok{ L.H }\OperatorTok{*}\NormalTok{ L }\OperatorTok{*}\NormalTok{ D }\OperatorTok{**}\NormalTok{ Rational(}\DecValTok{1}\NormalTok{,}\DecValTok{2}\NormalTok{)}
    \ControlFlowTok{return}\NormalTok{ D }\OperatorTok{**}\NormalTok{ Rational(}\DecValTok{1}\NormalTok{,}\DecValTok{2}\NormalTok{)}
\end{Highlighting}
\end{Shaded}

This algorithm converges for all real matrices (\cite{lrproof}). A nearly verbatim generaisation to double-complex matrices is provided by figure \ref{splitcomplex-svdcode}.

It turns out that the resulting algorithm is equivalent to the classic LR algorithm (\cite{lrproof,computations}) for eigendecomposition. In that sense, we've started with a special case of the LR algorithm (for SVD) and automatically derived the general LR algorithm (for eigendecomposition of the product of two matrices). The general LR algorithm is rarely used in practice due to its numerical instability.

\hypertarget{qr-decomposition-1}{%
\subsection{QR decomposition}\label{qr-decomposition-1}}

There are multiple ways of computing the QR decomposition of a complex matrix. Two such ways are via the Gram-Schmidt process (\cite{computations}) or Householder reflections (\cite{computations}). We have implemented both procedures and observed that they generalise to double matrices. In doing this, we have found some new algorithms for computing the LU decomposition of the product of two matrices. First though, we begin by presenting a \emph{third} way to compute the QR decomposition of a double matrix.

\hypertarget{qr-algorithm-via-decomposition-into-real-components}{%
\subsubsection{QR algorithm via decomposition into real
components}\label{qr-algorithm-via-decomposition-into-real-components}}

To compute the QR decomposition of a double matrix \([A,B]\), one
can break it down into the problem of finding:

\begin{itemize}
\tightlist
\item
  \(L\) and \(U\) such that \(BA = LU\)
\item
  \(C\) such that \(A = C U\) and \(B = LC^{-1}\).
\item
  Explicitly finding \(C^{-1}\).
\end{itemize}

The output would then be the two matrices \(R = [U,L]\) and
\(Q = [C,C^{-1}].\)

This can be done by carrying out the following steps:

\begin{enumerate}
\def\labelenumi{\arabic{enumi}.}
\tightlist
\item
  Compute \(BA\). This costs time \(n^3\).
  \footnote{We measure algorithm complexity by counting flops, which are the number of multiplications, additions, subtractions, divisions and square roots.}
\item
  Find the LU decomposition of \(BA\): \(LU = BA\). This costs time
  \(2n^3/3\).
\item
  Find \(C\) via the equation \(A = CU\) using forward substitution.
  This costs time \(n^3/2\)
\item
  Do the same to find \(C^{-1}\) using the quation \(B = L C^{-1}\).
  This costs time \(n^3/2\).
\end{enumerate}

The overall running time is thus \(8n^3/3\). This is the same as that of
the Gram-Schmidt process and Householder method.

It bears pointing out that finding the double matrix
\(Q = [C,C^{-1}]\) is optional in the above algorithm.

\hypertarget{gram-schmidt}{%
\subsubsection{Gram-Schmidt}\label{gram-schmidt}}

Here, we show how to implement the Gram-Schmidt process for
double matrices.

The running time of this algorithm is simply double that of the
Gram-Schmidt process on real matrices. That is, it is \(8n^3/3\).

See figure \ref{gramschmidt-code} for an implementation.

It bears pointing out that finding the matrix \(R = [U,L]\) is optional
in the above algorithm.

\hypertarget{householder-method}{%
\subsubsection{Householder method}\label{householder-method}}

The QR decomposition can also be computed using Householder reflections,
suitably generalised to double vectors. The running time is twice
that of the usual algorithm over real matrices, and is therefore equal
to \(8n^3/3\).

See figure \ref{householdercode} for an implementation.

\section{The role of involutions in the double-complex numbers}

The automorphism group of the double-complex algebra (understood as an algebra over the real numbers) has eight elements. Six of those have order $2$, and can therefore be called \emph{involutions}. Involutions are central to the study of double-complex matrices because they can be used as a substitute or even generalisation of complex conjugation. Linear algebra essentially reduces to just five operations: Addition, subtraction, multiplication, division, and conjugation.

We now observe that four of those involutions result in a trivial theory. They are:
\begin{itemize}
  \item $(w,z) \mapsto (w,z)$
  \item $(w,z) \mapsto (\overline w,z)$
  \item $(w,z) \mapsto (w,\overline z)$
  \item $(w,z) \mapsto (\overline w,\overline z)$
\end{itemize}
To illustrate why they result in a trivial theory, we pick one of them, and hope that the reader understands that a similar issue is present in the other three. By $(w,z)$, we mean $w \frac{1+j}2 + z\frac{1j}2$ where $w$ and $z$ are complex numbers.

For the sake of argument, we will define $(w,z)^*$ to mean $(\overline w,\overline z)$. We can define $(A,B)$ to mean $(1,0) A + (0,1) B$ where $A$ and $B$ are arbitrary complex matrices. What we observe is that if $f(M)$ is some matrix operation defined on complex matrices $M$, we can generalise its definition to double-complex matrices by interpreting any complex conjugations in its definition to mean the operation $(w,z)^* = (\overline w, \overline z)$. We will then immediately have that $f((A,B)) = (f(A),f(B))$. For instance, if $f$ is the pseudoinverse of a matrix (its verbatim generalisation that uses this particular involution) then we have that this operation always distributes over the components $A$ and $B$. The same thing is true if $f$ is the SVD or the LU decomposition, or we would argue \emph{anything}. There is therefore nothing to study.

We now give the only two involutions which are of interest. They are:
\begin{itemize}
  \item $(w,z) \mapsto (z,w)$
  \item $(w,z) \mapsto (\overline z, \overline w)$
\end{itemize}
We claim that these result in an identical theory to eachother. To understand why, we endeavour to define an operation $(A,B) \mapsto [A,B]$ on pairs of complex matrices, mapping to the double-complex matrices bijectively, that satisfies:
$$\begin{aligned}
[A,B] + [C,D] &= [A+C,B+D] \\
[A,B] \times [C,D] &= [AC, DB] \\
[A,B]^* &= [B,A]
\end{aligned}$$
The third identity involves an involution, so our definition of $[A,B]$ will change along with the involution. If the involution is defined as $(w,z) \mapsto (z,w)$, then we can define $[A,B]$ to mean $A \frac{1+j}2 + B^T\frac{1j}2$ as we have chosen to do in this paper. This then causes the above three identities to all be satisfied. Likewise, if we define our involution to be $(w,z) \mapsto (\overline z, \overline w)$, then we can define $[A,B]$ to mean $A \frac{1+j}2 + \overline B^T \frac{1-j}2$. Once again, the above three identities get satisfied.

So what's the difference between these two involutions? It is about which double-complex matrices $[A,B]$ are understood to be the complex ones. One approach makes it so that the double-complex matrices of the form $[A,A^T]$ are taken to be the complex ones. The other approach makes it so that double-complex matrices of the form $[A,A^*]$ are taken to be the complex ones. Since in this paper, we treat double-complex matrices as pairs of complex matrices, we don't ultimately care which ones are taken to be complex. By accident, we made it so that the complex matrices are those of the form $[A,A^T]$.

Observe that given either definition of $[A,B]$, the claim that a double-complex matrix is Hermitian precisely when it satisfies $A = B$ is still true. All the other characterisations of the different families of matrices given in section \ref{sec:families} stay the same.

The convention that matrices of the form $[A,\overline A^T]$ are the complex one (which is the one we didn't pick) is arguably slightly better. It results in the notions of unitary or Hermitian double-complex matrices being direct generalisations of their complex counterparts. For that reason, in future papers, this convention might be a slightly better one to adopt.

Given that it turns out that our notion of Hermitian matrix is not a direct generalisation of the one over complex matrices, what does it generalise? It turns out that it includes all the \emph{complex symmetric} matrices, which are those complex matrices which satisfy $A = A^T$. Our notion of unitary matrix generalises the set of \emph{complex orthogonal} matrices, which are those that satisfy $A^T A = AA^T = I$. This all follows from section \ref{sec:families}. Perhaps this suggests that complex orthogonality and unitarity are on equal footings with eachother.

\hypertarget{jordan-svd}{%
\section{Jordan SVD}\label{jordan-svd}}

We will begin by doing what we did in section \ref{simple-matrix-decompositions-over-the-double-numbers}, but for the polar decomposition. The polar decomposition of a square complex matrix $M$ is a decomposition of the form $M = UP$ where $U$ is unitary and $P$ is Hermitian.

Generalising the polar decomposition to the double-complex matrices gives us:
$$[A,B] = [P,P^{-1}][C,C]$$
which expands into:
$$\begin{aligned}A &= PC \\ B &= CP^{-1} \end{aligned}$$
This implies that:
$$\begin{aligned}BA &= C^2 \\ AB &= P C^2 P^{-1} \end{aligned}$$
We see that the double-complex polar decomposition of $[A,B]$ gives us a complex matrix $\sqrt{BA}$ and a matrix $P$ such that $P\sqrt{BA}P^{-1} = \sqrt{AB}$.

In this section, we will show that such a decomposition of $[A,B]$ exists whenever $A$ and $B$ are non-singular. The problem of determining precisely when a polar decomposition exists (in the singular case) remains an open problem.

Recall that over the complex numbers, SVD is equivalent to polar decomposition. Given the derivation of double-complex SVD in section \ref{simple-matrix-decompositions-over-the-double-numbers}, it follows that there are non-singular double-complex matrices which don't have SVDs. This follows from the fact that we need \(AB\) and \(BA\) to be diagonalisable. Due to this, the polar decomposition is not equivalent to SVD over double-complex matrices. We restore the equivalence of polar decomposition and SVD by defining a more sophisticated generalisation of the SVD to double-complex matrices called the Jordan SVD, which borrows features from the Jordan Normal Form. We proceed to show that it has uses in non-Euclidean geometry (section \ref{using-the-jordan-svd-to-challenge-a-claim-made-in-yagloms-complex-numbers-in-geometry}), the computation of Moore-Penrose pseudoinverses of double-complex matrices (section \ref{pseudoinverse-and-its-relation-to-jordan-svd}), and in generalising some well-known or obscure matrix decompositions (section \ref{jordan-svd-generalises}).

\begin{definition}[Jordan SVD]
The Jordan SVD of a double-complex matrix $M$ is a factorisation of the form $M = U [J,J] V^*$ where $J$ is a complex Jordan matrix, and $U$ and $V$ are unitary double-complex matrices.
\end{definition}

\begin{definition}[Polar decomposition] The \emph{polar decomposition} of a double-complex matrix $M$ is a factorisation of the form $M = UP$ where $U$ is unitary and $P$ is Hermitian. \end{definition}

\begin{theorem} A double-complex matrix $M$ has a Jordan SVD if and only if it has a polar decomposition. \end{theorem}
\begin{proof}
Assume that $M$ has a Jordan SVD. Write it as $W[J,J]V^*$. Let $U = WV^*$ and $P = V[J,J]V^*$. Observe that $UP$ is a polar decomposition of $M$.

We now prove the converse. Assume that $M$ has a polar decomposition $UP$. Write $P$ as $[A,A]$. $A$ has a Jordan decomposition $QJQ^{-1}$. We thus have that $M = U[Q,Q^{-1}][J,J][Q,Q^{-1}]^*$, which is a Jordan SVD where the three factors are $U[Q,Q^{-1}]$, $[J,J]$ and $[Q,Q^{-1}]^*$.
\end{proof}

Now we show that unlike the SVD, the Jordan SVD / polar decomposition always exists for non-singular matrices.

\begin{theorem}[Existence of Jordan SVD] \label{existence-jordan-svd}
If a double-complex matrix $M = [A,B]$ is invertible, then it has a Jordan SVD.
\end{theorem}

\begin{proof}
Since $M$ is invertible, so are $A$ and $B$. Since $A$ and $B$ are invertible, we have that $\sqrt{AB}$ exists. Consider the Jordan decomposition of $\sqrt{AB}$, which we will write as $PJP^{-1}$. We have that $AB = P J^2 P^{-1}$. Let $Q = A^{-1} P J$. We can therefore express $A$ in two ways: $A = PJQ^{-1} = P J^2 P^{-1} B^{-1}$. Cancelling $P$ and $J$ gives $Q^{-1} = J P^{-1} B^{-1}$. Rearranging gives $B = Q J P^{-1}$.

We can therefore express $A$ and $B$ as $PJQ^{-1}$ and $QJP^{-1}$ respectively. Let $U = [P,P^{-1}]$ and $V = [Q,Q^{-1}]$. We have that $M = U [J,J] V^*$, as claimed.
\end{proof}

By the \emph{half-plane}, we mean the set of all those complex numbers
which have real part greater than zero, together with all complex
numbers whose real part equals zero but whose imaginary part is
non-negative. Denote the half-plane as \(H\). In symbols, we have that
\(H = \{z \in \mathbb C : \Re(z) > 0\} \cup \{ix : x \in \mathbb R, x \geq 0\}\).

We aim to show now that as long as a Jordan SVD of a matrix exists, then
there exists a unique Jordan SVD in which all the eigenvalues of \(J\)
are on the half-plane (up to permutation of the Jordan blocks of \(J\)).

\begin{lemma} \label{existence-half-plane} If a matrix $M = [A,B]$ has a Jordan SVD $U [J,J] V^*$, then it has a Jordan SVD $U' [J',J'] (V')^*$ where each element on the diagonal of $J'$ belongs to the half-plane. \end{lemma}

\begin{proof}

Write $J = J_1 \oplus J_2 \oplus \dotsc \oplus J_k$ where each $J_i$ denotes some Jordan block of $J$. Call the corresponding eigenvalues $\lambda_1, \lambda_2, \dotsc, \lambda_k$. Assume for the sake of illustration that $\Re(\lambda_1) < 0$ and for the rest we have $\Re(\lambda_i) > 0$. Furthermore, assume that $J_1$ has dimensions $m \times m$ for some $m$. We observe that $(-I_{m}) \oplus I_{n-m}$ multiplied by $J$ changes $J_1$ to a non-negative block. We also observe that $U ((-I_{m}) \oplus I_{n-m})$ is unitary. We thus have that $M$ is equal to $U' [J',J'] V^*$ where $U' = U ((-I_{m}) \oplus I_{n-m})$ and $J' = ((-I_{m}) \oplus I_{n-m}) J$. We have that $J'$ is not necessarily a Jordan matrix, but is still similar to a Jordan matrix which has the same eigenvalues as it. We take the Jordan decomposition of $J'$ to get $J' = P J'' P^{-1}$. We thus have that $[J',J'] = [PJ''P^{-1}.PJ''P^{-1}] = [P,P^{-1}] [J'',J''] [P,P^{-1}]^*$. Putting it all together we have that $M$ can be expressed as $U'' [J'',J''] (V')^*$ where $U'' = U' [P^{-1}, P]$, and $J''$ is defined as earlier, and $V' = V [P,P^{-1}]$.
\end{proof}

The following lemma characterises the possible Jordan Normal Forms of
the square roots of any invertible complex matrix.

\begin{lemma} \label{jnf-square-root} Let $M$ be an invertible complex matrix. Consider its Jordan Normal Form $J = J_{k_1}(\lambda_1) \oplus J_{k_2}(\lambda_2) \oplus \dotsb \oplus J_{k_n}(\lambda_n)$, where $n$ is the number of Jordan blocks in $J$, and $k_i$ denotes the size of the $i$th Jordan block. The Jordan Normal Form of any square root of $M$ is of the form $J_{k_{\sigma(1)}}(\pm \sqrt {\lambda_{\sigma(1)}}) \oplus J_{k_{\sigma(2)}}(\pm\sqrt{\lambda_{\sigma(2)}}) \oplus \dotsb \oplus J_{k_{\sigma(n)}}(\pm\sqrt{\lambda_{\sigma(n)}})$ where $\sigma$ is some permutation on $\{1,2,\dotsc,n\}$. \end{lemma}

\begin{proof}

We have that $M = P J P^{-1}$ where $J$ is some Jordan matrix. We have that either $J = J_{k_1}(\lambda_1) \oplus J_{k_2}(\lambda_2) \oplus \dotsb \oplus J_{k_n}(\lambda_n)$.

Consider some arbitrary square root of $M$, and denote it as $\sqrt M$. Consider the Jordan decomposition of $\sqrt M$, which we shall write as $Q K Q^{-1}$. We have that $K = J_{\ell_1}(\mu_1) \oplus J_{\ell_2}(\mu_2) \oplus \dotsb \oplus J_{\ell_m}(\mu_m)$ where $m$ is the number of Jordan blocks in $K$, and $\ell_i$ denotes the size of the $i$th Jordan block of $K$. We have that $M = Q K^2 Q^{-1}$. Consider the Jordan decomposition of $K^2$, which we will write as $R K' R^{-1}$. We get that $M = (QR) K' (QR)^{-1}$. Since the Jordan Normal Form of $J_{\ell_i}(\mu_i)^2$ is $J_{\ell_i}(\mu_i^2)$, we have that $K' = J_{\ell_1}(\mu_1^2) \oplus J_{\ell_2}(\mu_2^2) \oplus \dotsb \oplus J_{\ell_m}(\mu_m^2)$. Since both $K'$ and $J$ are the Jordan Normal Forms of $M$, we have that $m = n$, and there must exist a permutation $\sigma$ such that $\mu_i^2 = \lambda_{\sigma(i)}$ and $\ell_i = k_{\sigma(i)}$. The conclusion follows.
\end{proof}

We now prove that the Jordan SVD is unique as long as the eigenvalues of
\(J\) are all on the half-plane.

\begin{theorem} \label{existence-and-uniqueness-jsvd} An invertible double-complex matrix $M = [A,B]$ always has a Jordan SVD $U[J,J]V^*$ where the eigenvalues of $J$ are on the half-plane, and this is unique. \end{theorem}

\begin{proof}

By theorem \ref{existence-jordan-svd}, $M$ has a Jordan SVD $U [J,J] V^*$.

Let $J = J_{k_1}(\lambda_1) \oplus J_{k_2}(\lambda_2) \oplus \dotsb \oplus J_{k_n}(\lambda_n)$. We have that $AB = P J^2 P^{-1}$. Let $K$ be the Jordan Normal Form of $AB$. $K$ must also the Jordan Normal Form of $J^2$, so we get that $K = J_{k_1}(\lambda_1^2) \oplus J_{k_2}(\lambda_2^2) \oplus \dotsb \oplus J_{k_n}(\lambda_n^2)$. Consider any other Jordan SVD of $M$, which we will write as $[P',(P')^{-1}] [J',J'] (V')^*$. We have that $J'$ must be the Jordan Normal Form of some square root of $AB$, namely the one that's equal to $P' J' (P')^{-1}$. It then follows from lemma \ref{jnf-square-root} and the form of $K$ that $J' = J_{k_{\sigma(1)}}(\pm\lambda_{\sigma(1)}) \oplus J_{k_{\sigma(2)}}(\pm\lambda_{\sigma(2)}) \oplus \dotsb \oplus J_{k_{\sigma(n)}}(\pm\lambda_{\sigma(n)})$ (where $\sigma$ is some permutation on $\{1,2,\dotsc,n\}$). For each $i$, our choice of plus or minus is fully determined by our desire for $\pm\lambda_{\sigma(i)}$ to land on the half-plane. In detail: If the real part of $\lambda_{\sigma(i)}$ is negative, then we pick the minus option to make the real part positive. If the real part of $\lambda_{\sigma(i)}$ is positive, then we pick the plus option to make the real part positive. The remaining case is where $\Re(\lambda_{\sigma(i)}) = 0$; in that case, pick plus or minus depending on whichever one lands in the halfplane.

From the above paragraph, we conclude that there exists at most one value of $J'$ (up to permutation of the Jordan blocks) in which all the eigenvalues are on the half-plane. By lemma \ref{existence-half-plane}, this value of $J'$ exists. We are done.
\end{proof}

We finish this section by giving the Sympy code for computing the Jordan
SVD under the assumption that a double-complex matrix is invertible:

\begin{Shaded}
\begin{Highlighting}[]
\KeywordTok{def}\NormalTok{ jordan\_svd(M):}
\NormalTok{    A, B }\OperatorTok{=}\NormalTok{ first\_part(M), second\_part(M.T)}
\NormalTok{    P, J }\OperatorTok{=}\NormalTok{ Matrix.jordan\_form((A }\OperatorTok{*}\NormalTok{ B) }\OperatorTok{**}\NormalTok{ Rational(}\DecValTok{1}\NormalTok{,}\DecValTok{2}\NormalTok{))}
\NormalTok{    Q }\OperatorTok{=}\NormalTok{ A }\OperatorTok{**} \OperatorTok{{-}}\DecValTok{1} \OperatorTok{*}\NormalTok{ P }\OperatorTok{*}\NormalTok{ J}
\NormalTok{    U }\OperatorTok{=}\NormalTok{ P }\OperatorTok{*}\NormalTok{ (}\DecValTok{1} \OperatorTok{+}\NormalTok{ j)}\OperatorTok{/}\DecValTok{2} \OperatorTok{+}\NormalTok{ P.T }\OperatorTok{**} \OperatorTok{{-}}\DecValTok{1} \OperatorTok{*}\NormalTok{ (}\DecValTok{1} \OperatorTok{{-}}\NormalTok{ j)}\OperatorTok{/}\DecValTok{2}
\NormalTok{    V }\OperatorTok{=}\NormalTok{ Q }\OperatorTok{*}\NormalTok{ (}\DecValTok{1} \OperatorTok{+}\NormalTok{ j)}\OperatorTok{/}\DecValTok{2} \OperatorTok{+}\NormalTok{ Q.T }\OperatorTok{**} \OperatorTok{{-}}\DecValTok{1} \OperatorTok{*}\NormalTok{ (}\DecValTok{1} \OperatorTok{{-}}\NormalTok{ j)}\OperatorTok{/}\DecValTok{2}
\NormalTok{    S }\OperatorTok{=}\NormalTok{ J }\OperatorTok{*}\NormalTok{ (}\DecValTok{1} \OperatorTok{+}\NormalTok{ j)}\OperatorTok{/}\DecValTok{2} \OperatorTok{+}\NormalTok{ J.T }\OperatorTok{*}\NormalTok{ (}\DecValTok{1} \OperatorTok{{-}}\NormalTok{ j)}\OperatorTok{/}\DecValTok{2}
    \ControlFlowTok{return}\NormalTok{ U, S, V}
\end{Highlighting}
\end{Shaded}

\subsection{As a generalisation of other matrix decompositions} \label{jordan-svd-generalises}

The Jordan SVD generalises the Jordan decomposition. Let $PJP^{-1}$ be the Jordan decomposition of a complex matrix $A$. We then have that $[A,A] = U [J,J] V^*$ where $U = V = [P,P^{-1}]$. Note that $[A,A]$ is the general form of a double-complex Hermitian matrix.

Observe that the Jordan SVD generalises the ordinary SVD of a complex matrix. Let $UDV^*$ be the SVD of a complex matrix $A$. A Jordan SVD of $[A,\overline A^T]$ is then $[U,\overline U^T][D,D][V,\overline V^T]^*$.

Observe that the double-complex polar decomposition generalises the so-called \emph{Algebraic Polar Decomposition} (\cite{algebraicpolar,kaplansky}) of complex matrices. The Algebraic Polar Decomposition of a complex matrix $A$ is a factorisation of the form $A = UP$ where $U$ satisfies $UU^T = I$ and $P$ satisfies $P = P^T$. Given $M = [A,A^T]$, we have that a polar decomposition of $M$ is of the form $[U,U^T] [P,P]$.

\hypertarget{double-jordan-svd-as-opposed-to-double-complex-jordan-svd}{%
\subsection{Double-number Jordan SVD as opposed to double-complex Jordan
SVD}\label{double-jordan-svd-as-opposed-to-double-complex-jordan-svd}}

Recall that there is an alternative to the Jordan Normal Form over real
matrices where Jordan blocks of the form
\(\begin{bmatrix}a + bi & 0 \\ 0 & a - bi\end{bmatrix}\) are changed to
blocks of the form \(\begin{bmatrix}a & -b \\ b & a\end{bmatrix}\). This
change results in a variant of the Jordan Normal Form that uses only
real numbers and not complex numbers. Using this variant of the JNF,
it's possible to define a variant of the Jordan SVD for double
matrices without using double-complex numbers. For the sake of completeness,
we define this variant now but don't use it later on.

\begin{definition}[Double-number Jordan SVD]
The double-number Jordan SVD of a double matrix $M$ is a factorisation of the form $M = U [J,J] V^*$ where $J$ is a real Jordan matrix, and $U$ and $V$ are unitary double matrices.
\end{definition}

\hypertarget{using-the-jordan-svd-to-challenge-a-claim-made-in-yagloms-complex-numbers-in-geometry}{%
\section{\texorpdfstring{Using the Jordan SVD to challenge a claim made
in Yaglom's \emph{Complex Numbers in
Geometry}}{Using the Jordan SVD to challenge a claim made in Yaglom's Complex Numbers in Geometry}}\label{using-the-jordan-svd-to-challenge-a-claim-made-in-yagloms-complex-numbers-in-geometry}}

\hypertarget{introduction-to-complex-numbers-in-geometry}{%
\subsection{\texorpdfstring{Introduction to \emph{Complex Numbers in
Geometry}}{Introduction to Complex Numbers in Geometry}}\label{introduction-to-complex-numbers-in-geometry}}

We work with the English translation of Yaglom's \emph{Complex Numbers
in Geometry}, published in 1968. We will abbreviate this to \emph{CNG}.

The topic of CNG is \emph{linear fractional transformations} over
different number systems. A linear fractional transformation (from now
on, abbreviated to LFT) is a function of the form
\(z \mapsto \frac{az + b}{cz + d}\) where all variables \(a, b, c, d\)
and \(z\) belong to some number system (formally, some commutative
ring). Any such transformation can be represented using the
\(2 \times 2\) matrix \(\begin{bmatrix}a & b \\ c & d\end{bmatrix}\).
If a non-singular \(2 \times 2\) matrix \(M\) represents an LFT, then any
scalar multiple \(\lambda M\) of \(M\) represents the same LFT, as long
as \(\lambda\) is not a zero-divisor.

CNG considers three number systems: The complex numbers, dual numbers
and the double numbers. LFTs over complex numbers are well known.
They are usually called \emph{Moebius transformations}.

CNG provides some motivation for studying LFTs in the case of dual
numbers and double numbers, a topic which some might otherwise
find esoteric. To do this, it presents a geometric interpretation of
what elements in the dual numbers, double numbers and complex
numbers represent.

We will begin with the dual numbers, because it is the easiest and most
fundamental case. Naively, one might think that a dual number represents
a point on a plane; that is, the dual number \(x + \epsilon y\)
represents the point on the Cartesian plane with Cartesian coordinate
\((x,y)\). While this of course can work, CNG presents another
interpretation of what a dual number represents. In CNG, a dual number
is understood to represent a \emph{line} on the plane. Additionally,
this line is more than just a line, in that it is also \emph{oriented},
meaning that the line is pointing in one of two opposite directions.
This way, every line on the plane is represented by \emph{two} dual
numbers, depending on which of two orientations is given to the line.

A dual number of the form $\tan(\theta/2)(1 + \epsilon s)$ is understood to
represent a line whose $x$-intercept is $s$ and which makes an angle $\theta$
with the $x$ axis. Notice that if $\pi$ gets added to $\theta$, then the resulting
dual number is not the same, even though the line is the same. This shows that
the lines are oriented.

Strictly speaking, the set of dual numbers doesn't quite suffice to
express every line on the plane (because a line may not intersect the $x$ axis). This is a problem which can be fixed by extending the set of dual numbers with the addition of some infinite dual numbers. Those additional numbers are of the form \(\frac{1}{x\epsilon}\) where \(x\) can be an arbitrary real number. To understand this topic in complete generality, we suggest looking up \emph{homogeneous coordinates} and \emph{projective lines over rings}. Having made this change, as CNG does, we can now represent every line on the plane. Using homogeneous coordinates, a line which makes an angle $\theta$ with the $x$ axis, and which has a perpendicular distance of $R$ from the origin, can be represented by the point on the dual-number projective line with homogeneous coordinate $\left[\sin\left(\frac{\theta + \epsilon R} 2\right):\cos\left(\frac{\theta + \epsilon R} 2\right)\right]$.

The elements \(a, b, c, d\) in \(z \mapsto \frac{az + b}{cz + d}\) are
all ordinary dual numbers, while \(z\) can also take any value of the
form \(\frac{1}{x\epsilon}\).

CNG shows that the LFTs over the dual numbers contain all Euclidean
isometries. What this means is that all translations, rotations and
reflections can be expressed as LFTs over the dual numbers. Strictly
speaking, the reflections are always followed by \emph{orientation
reversals}, implying that while there is a subgroup of dual-number LFTs
that's isomorphic to the Euclidean group, the geometric interpretation
of this subgroup is somewhat inconsistent with how one might normally
imagine the Euclidean group. This subgroup is represented by the dual-number
equivalent of the unitary matrices.

The LFTs over the dual numbers contain some transformations which are
not Euclidean isometries. Because of this, Yaglom is motivated to
classify all possible LFTs over the dual numbers. In a paper by Gutin,
this classification is interpreted as a \emph{matrix decomposition} of
dual number matrices. Understood geometrically, matrix decompositions endeavour
to express arbitrary matrices as products of simpler matrices. This
is equivalent to describing a geometric transformation as a sequence of
simpler and sometimes more familiar transformations. Gutin shows that
Yaglom's decomposition is of the form \(USV^*\) where \(U\) and \(V\)
are unitary matrices over the dual numbers. All
of the exotic behaviours of the dual number LFTs are due to the
matrix \(S\). Gutin shows that \(S\) need only be of the form:

\begin{itemize}
\tightlist
\item
  \(\begin{bmatrix}a & 0 \\ 0 & b\end{bmatrix}\) where \(a\) and \(b\)
  are arbitrary real numbers.
\item
  \(\begin{bmatrix}a & -b\epsilon \\ b\epsilon & a \end{bmatrix}\) where
  \(a\) and \(b\) are arbitrary real numbers.
\end{itemize}

A similarity to Singular Value Decomposition is now obvious. Gutin goes
on to show that all matrices over the dual numbers can be expressed in a
similar way to the above (but generalised appropriately to
\(n \times n\) matrices for arbitrary \(n\), and to singular matries as
well). The above special case is only for \(2 \times 2\) matrices, and
only for invertible matrices. This special case is enough for CNG.

Note: If $S$ is of the form $\begin{bmatrix}a & -b\epsilon \\ b\epsilon & a \end{bmatrix}$ (which is the second case), then $S$ represents a transformation called an \emph{axial dilation}. Axial dilations commute with unitary matrices. As a result, we have that $USV^* = UV^* S$. In other words, if $S$ is an axial dilation then we can simplify the decomposition to $US$ instead of $USV^*$. In the axial dilation case, the book states the decomposition in this simpler way.

A question arises: What happens if the role of the dual numbers above is changed to the complex numbers? In that case, the complex numbers represent oriented lines in the \emph{elliptic plane} (the plane which elliptic geometry takes places over). This is in contrast to the dual numbers, which represent oriented lines in the Euclidean plane. The elliptic plane is essentially a sphere (but where antipodal points are identified), and the lines are thus great circles. An arbitrary great circle can be chosen to be the equator. The oriented great circle which intersects the equator at longitude $s$, and makes an angle $\theta$ with the equator at the point of intersection, can be represented by the complex number $\tan(\theta/2)(\cos(s) + i \sin(s))$. In the case where $\theta = \pi$ (where the line is the same as the equator, but oriented in the opposite direction to when $\theta = 0$) the oriented line is represented as $\infty$. Similar to the case of the dual numbers, the unitary matrices act as isometries of the elliptic plane. The set of transformations of the elliptic plane expressible as complex-number LFTs can be decomposed using Singular Value Decomposition of complex matrices, in a similar way to how we decomposed transformations of the Euclidean plane (expressible as dual-number LFTs) using an analogue of Singular Value Decomposition for dual-number matrices.

It's worth pointing out that CNG does \emph{not} interpret complex LFTs as transformations of the elliptic plane. We won't speculate on why the book omitted this. Instead, the book proposes some other interpretations of complex LFTs. In one such interpretation, the complex numbers are understood to represent \emph{oriented points} (oriented either clockwise or anti-clockwise) in the \emph{hyperbolic plane} (in the sense of hyperbolic geometry). We won't say any more about this.

Before we describe the double case, we must mention the fact that
there are three non-Euclidean metric geometries over the plane:
Euclidean geometry, elliptic geometry and hyperbolic geometry. The significance of these geometries lies in the fact that they satisfy four of the five postulates in Euclid's Elements (once they are formulated in a suitable manner) while the fifth postulate is contradicted by two of the three geometries. We've
related the dual numbers to Euclidean geometry and the complex numbers
to elliptic geometry. It's clear what comes next.

The double numbers represent oriented lines on the
\emph{hyperbolic plane}. Strictly speaking, one must projectively extend
the double numbers to represent all such lines, but the
preceeding claim is mostly correct. The LFTs over the double
numbers can be seen to contain a subgroup isomorphic to the group of all
isometries of the hyperbolic plane.

\hypertarget{a-challenge-to-yagloms-classification-of-lfts-over-the-double-numbers}{%
\subsection{A challenge to Yaglom's classification of LFTs over the
double
numbers}\label{a-challenge-to-yagloms-classification-of-lfts-over-the-double-numbers}}

Like in the dual number case, CNG endeavours to classify all possible
LFTs over the double numbers. We suspect that this classification
is not entirely correct.

Below, we provide the relevant quote from CNG, and do our best to
faithfully interpret it\footnote{But please see the last paragraph.}:

\begin{quote}
\emph{Each axial circular transformation of the Lobachevskii plane
represents a motion, or a motion together with an axial symmetry with
respect to some cycle \(S_1\) (axial inversion of the first, second, or
third kind), or a motion together with an axial symmetry with respect to
an equidistant curve and a reversion (that is, together with an axial
inversion of the fourth kind).}
\end{quote}

In our copy of the book, this can be found at the end of Section III,
just before the Appendix (page 194). By an \emph{axial circular
transformation of the Lobachevskii plane}, CNG means an LFT over the
double numbers. This is the term that the book uses in place of
double LFT. The term \emph{Lobachevskii plane} is used in some
books in place of hyperbolic plane.

A \emph{motion} is represented precisely by a unitary matrix over the
double numbers. CNG uses this to mean a hyperbolic isometry. We
make the daring choice to interpret CNG's classification as a
decomposition of the form \(U S V^*\) where \(U\) and \(V\) are unitary
matrices, and \(S\) is some matrix that CNG confines to only a few sets
of possibilities. This seems to be what CNG meant, because the resulting
set of possibilities for \(S\) is nearly exhaustive. Perhaps the choice
of \(U\) and \(V\) could be restricted depending on \(S\), but CNG's
classification has the best chance of being exhaustive if we allow \(U\)
and \(V\) to be chosen arbitrarily. Here are the possibilities CNG
allows for \(S\):

\begin{enumerate}
\def\labelenumi{\arabic{enumi}.}
\tightlist
\item
  By an axial inversion of the first kind, CNG means a matrix of the
  form \(\begin{bmatrix} 0 & -k \\ 1 & 0\end{bmatrix}\) where \(k\) is
  an arbitrary real number.
\item
  By an axial inversion of the second kind, CNG means a matrix of the
  form \(\begin{bmatrix}j & -1 \\ 3 & j\end{bmatrix}\). Oddly enough,
  this is a single matrix, and not an infinite family.
\item
  By an axial inversion of the third kind, CNG means a matrix of the
  form
  \(\begin{bmatrix}(1 - \alpha)j & 1 + \alpha \\ -(1 + \alpha) & (1 - \alpha)j \end{bmatrix}\),
  where \(\alpha\) is an arbitrary non-negative real number.
\item
  By an axial inversion of the fourth kind, CNG means a matrix of the
  form \(\begin{bmatrix}0 & kj \\ 1 & 0 \end{bmatrix}\).
\end{enumerate}

In this paper, we provide an alternative, and we would argue fairly
similar, decomposition of double matrices called the Jordan SVD.
We now proceed to use it.

If the above set of values for \(S\) is indeed exhaustive, then we can
check that claim by using the Jordan SVD. Recall that the Jordan SVD of
a double-complex matrix of the form \([A,B]\) is the decomposition
\([A,B] = U [J,J] V^*\) where \(U\) and \(V\) are unitary matrices over
the double-complex numbers and \(J\) is a complex Jordan matrix (the same as
in the Jordan Normal Form of a complex matrix). To check exhaustiveness,
we only need to check that all possible values of \(J\) are exhausted,
where \(J\) is a possible Jordan Normal Form of a real matrix. Here are
the matrices \(J\) for the above four possible sets of values of \(S\):

\begin{enumerate}
\def\labelenumi{\arabic{enumi}.}
\tightlist
\item
  \(\begin{bmatrix} k & 0 \\ 0 & 1\end{bmatrix}\) for some
  \(k \in \mathbb R\).
\item
  \(\begin{bmatrix} 2 & 1 \\ 0 & 2\end{bmatrix}\).
\item
  \(\left[\begin{matrix}\sqrt{- 2 i \alpha^{2} + 4 \alpha + 2 i} & 0\\0 & \sqrt{2 i \alpha^{2} + 4 \alpha - 2 i}\end{matrix}\right]\).
\item
  \(\begin{bmatrix}1 & 0 \\ 0 & ik \end{bmatrix}\).
\end{enumerate}

Recall that if a matrix \(M\) represents an LFT, then \(\lambda M\)
represents the same LFT as long as \(\lambda\) is not a zero divisor.
Taking this into account, we can generalise the above possibilities to:

\begin{enumerate}
\def\labelenumi{\arabic{enumi}.}
\tightlist
\item
  \(\begin{bmatrix} a & 0 \\ 0 & b\end{bmatrix}\) where \(a\) and \(b\)
  are arbitrary real numbers. This is obtained by taking case 1,
  substituting \(k = a/b\) and multiplying by \(\lambda = b\).
\item
  \(\begin{bmatrix} 2 & 1 \\ 0 & 2\end{bmatrix}\). Taking case 2 and
  scaling the matrix by some \(\lambda\) does not result in a Jordan
  matrix unless \(\lambda = 1\). This appears to be the downfall of
  CNG's claim that this is an exhaustive classification.
\item
  \(\begin{bmatrix} z & 0 \\ 0 & z^* \end{bmatrix}\). This can be
  obtained by taking case 3, substituting in the appropriate value of
  \(\alpha\) and multiplying by the appropriate \(\lambda\).
\item
  \(\begin{bmatrix} c\frac{1 - i}{\sqrt 2} & 0 \\ 0 & c\frac{1 + i}{\sqrt 2} \end{bmatrix}\),
  which is obtained by taking case 4, substituting \(k = 1\) and
  multiplying by \(c\frac{1 - i}{\sqrt 2}\).
\end{enumerate}

The LFT of the form \(\frac{2z+(1+j)}{(1-j)z + 2}\) doesn't appear to be
covered by the above classification. The corresponding matrix is
\([J,J]\) where \(J = \begin{bmatrix} 1 & 1 \\ 0 & 1 \end{bmatrix}\).
This matrix isn't covered by any of the four cases above. This suffices
to show that the classification in CNG is incomplete.

One proposal for completing CNG's classification is to replace case 2,
which CNG calls an \emph{axial inversion of the second kind} (of which
there is only one), with the infinite family of matrices of the form
\(\begin{bmatrix}k & 1 + j \\ 1 - j & k \end{bmatrix}\) where \(k\) is
an arbitrary real number. The only objection to this might be that CNG's
classification is of a geometric nature, and it's unclear what geometric
meaning such matrices have. Further work might be needed to find a
geometrically meaningful extension to case 2.

We must point out that we have slightly simplified matters above. This
was done for the sake of clarity, but not in a way that we think
undermines CNG's claim. When CNG studies double LFTs, it allows the use
of the conjugation operation: \(z \mapsto z^*\). Therefore the group
that it studies consists of transformations of the form
\(z \mapsto \frac{az + b}{cz + d}\) and
\(z \mapsto \frac{az^* + b}{cz^* + d}\). The latter transformations
don't admit a matrix representation. Furthermore, the four different
types of \emph{axial inversions} (as described in the book) are all transformations of the latter type, and therefore don't have matrix representations. This changes none
of our conclusions. Why? We interpret CNG's claim to be that every
``axial circular transformation of the Lobachevskii plane'' \(T\) is of
the form \(U \circ S \circ V\) where \(U\) is a motion, \(S\) is an
axial inversion of one of the four types, and \(V\) is another motion. Let
\(C\) denote conjugation. If \(T\) reverses orientations and
\(T = U \circ S \circ V\), we have that
\(T = U' \circ S' \circ V' \circ C\) where \(U'\) corresponds to a
unitary matrix, \(S'\) corresponds to one of the four types of matrices
we associated with axial inversions, and \(V'\) corresponds to another
unitary matrix. Our conclusions are therefore unchanged.

\hypertarget{pseudoinverse-and-its-relation-to-jordan-svd}{%
\section{Pseudoinverse, and its relation to Jordan
SVD}\label{pseudoinverse-and-its-relation-to-jordan-svd}}

It is natural to consider an analogue of the Moore-Penrose pseudoinverse
over double-complex matrices. Over the real numbers and complex numbers, the
pseudoinverse has a connection to the SVD: Namely, that
\((USV^*)^+ = VS^+U^*\), where \(M^+\) means the pseudoinverse of a
matrix \(M\). This connection is not as useful over double or
double-complex matrices, because a pseudoinverse might exist while a (naive)
SVD might not. The Jordan SVD remedies this, and re-establishes the
connection between SVD and pseudoinverse.

In a twin paper to this one, by Gutin, entitled \emph{An analogue of the
relationship between SVD and pseudoinverse over double-complex matrices}, it
is shown that whenever a double-complex matrix has a pseudoinverse then it
also has a Jordan SVD. This means that the Jordan SVD can, in principle,
always be used to find the pseudoinverse of a double-complex matrix: Let the
double-complex matrix \(M\) have Jordan SVD \(U[J,J]V^*\). Assuming that
\(J\) has no non-trivial nilpotent Jordan blocks, we have that
\(M^+ = V[J^+,J^+]U^*\).

In the same paper, it is also shown that a sufficient condition for
\([A,B]\) to have a Jordan SVD is for
\(\operatorname{rank}(A) = \operatorname{rank}(B) = \operatorname{rank}(AB) = \operatorname{rank}(BA)\).
This is also a necessary and sufficient condition for \([A,B]\) to have
a pseudoinverse.

\hypertarget{extending-the-set-of-permutation-matrices}{%
\section{Extending the set of permutation
matrices}\label{extending-the-set-of-permutation-matrices}}

In order for a decomposition to exist in general (or be computable in a
stable manner) it is often necessary to introduce \emph{pivoting}. In
this section, we investigate a hypothesis on how to extend such
decompositions to double matrices.

First we must ask what the analogue is of a permutation matrix over the
double numbers. It is tempting to say it is of the form
\([P,P^{-1}]\) where \(P\) is a permutation matrix over the real
numbers. The justification for this claim is that such a matrix is
necessarily a real matrix (because \(P^{-1} = P^T\)). Now we argue that
this is undesirable, and that a much nicer analogue is the set of
\emph{pairs} of permutation matrices \([P,Q]\). We will first assume,
for the sake of argument, that a permutation matrix should be of the
form \([P,P^{-1}]\):

Consider the LUP decomposition: \emph{Every} complex matrix \(M\) can be
written in the form \[M = LUP\] where \(L\) is a lower triangular
matrix, \(U\) is an upper triangular matrix, and \(P\) is a permutation
matrix. We would like to extend the above claim to double
matrices as well. So we get: \[[A,B] = [L_1,U_1] [U_2, L_2] [P,P^{-1}]\]
which therefore results in the claim that we can always solve
\[A = L_1 U_2 P, B = P^{-1} L_2 U_1.\] We now demonstrate that this is
not the case: Consider
\(A = \begin{pmatrix}1 & 0 \\ 0 & 1 \end{pmatrix}\) and
\(B = \begin{pmatrix} 0 & 1 \\ 1 & 0 \end{pmatrix}\). By the existence
criteria for LU decomposition (\cite{computations}), the above factorisation
of \(A\) and \(B\) is impossible.

To remedy this problem, we extend the set of permutation matrices to the
double numbers in a different way. We define a
\emph{double permutation matrix} to be of the form \([P,Q]\)
where \(P\) and \(Q\) are arbitrary permutation matrices. The analogue
of LUP decomposition reduces to the trivial but correct claim that we
can take the LUP decomposition separately of two real matrices.

\hypertarget{decompositions-that-feature-permutation-matrices}{%
\section{Decompositions that feature permutation
matrices}\label{decompositions-that-feature-permutation-matrices}}

In this section, we study double analogues of matrix
decompositions that feature permutation matrices. We observe that once
we take components, we arrive at familiar decompositions of real
matrices.

\hypertarget{bkp-decomposition}{%
\subsection{BKP decomposition}\label{bkp-decomposition}}

We use the name \emph{BKP decomposition} (where BKP stands for
\emph{Bunch-Kaufman-Parlett}) for a decomposition first introduced in a
paper by Bunch and Parlett (\cite{bunchparlett}), and then further studied in
a paper by Bunch and Kaufman (\cite{bunchkaufman}).

The \emph{BKP decomposition} (\cite{computations}) is a decomposition of
Hermitian matrices \(M\) that takes the form \[PMP^* = LDL^*\] where
\(P\) is a permutation matrix, \(L\) is a lower triangular matrix, and
perhaps surprisingly, \(D\) is a \emph{block-diagonal} Hermitian matrix
where the blocks are of size \(1\times1\) or \(2\times2\).

The analogue of the above for double matrices is the
decomposition \[[P,Q] [A,A] [Q,P] = [L,U] [D,D] [U,L].\] We take
components to get \[PAQ = LDU\] which is a variant of the LUP
decomposition where the matrix \(D\) is block-diagonal with \(1\times1\)
or \(2\times2\) blocks.

\hypertarget{rrqr-decomposition}{%
\subsection{RRQR decomposition}\label{rrqr-decomposition}}

The \emph{RRQR decomposition} (for \emph{rank-revealing QR
decomposition}, \cite{computations,rrqr}) is a decomposition of real matrices
\(M\) that takes the form \[M\Pi = QR\] where \(\Pi\) is a permutation
matrix, \(Q\) is a unitary matrix, and \(R\) is an upper triangular
matrix.

Consider the double analogue of the above
\[[A,B] [\Pi_1, \Pi_2] = [C,C^{-1}][U,L]\] which breaks apart into
\[A \Pi_1 = CU, \Pi_2 B = LC^{-1}.\] Observe that
\(\Pi_2 B A \Pi_1 = LU\). In other words, we get an LUP decomposition of
the product of two matrices \(B\) and \(A\).

\section{Difficulties automatically extending numerical algorithms
to matrix decompositions featuring permutation matrices}

We briefly note that most algorithms that compute decompositions featuring permutation matrices do not straightforwardly generalise to our decompositions over double matrices. This includes all algorithms based on variants of Gaussian elimination with pivoting. It appears that those algorithms make rigid assumptions about the set of permutation matrices. We have attempted to find a method for carrying out such a generalisation, but are not satisfied with our results.

\hypertarget{conclusion}{%
\section{Open problems}\label{conclusion}}

\subsection{Jordan SVD}

We finish this paper with two open problems concerning the Jordan SVD.
The Jordan SVD can be defined without using double or double-complex
numbers, allowing the broader linear algebra community to suggest
solutions to these problems. We say that a pair of \(n \times n\)
complex matrices \(A\) and \(B\), written \([A,B]\), has a Jordan SVD if
there exist invertible matrices \(P\) and \(Q\), and a Jordan matrix
\(J\), such that:

\begin{itemize}

\item
  \(A = PJQ^{-1}\).
\item
  \(B = QJP^{-1}\).
\end{itemize}

The two open problems are:

\begin{enumerate}
\def\labelenumi{\arabic{enumi}.}

\item
  Find a necessary and sufficient condition for the Jordan SVD to exist.
\item
  Prove that the Jordan matrix \(J\) is unique in all cases, subject to
  the restriction that all the eigenvalues of \(J\) belong to the
  \emph{half-plane}. We defined the half-plane to consist of those
  complex numbers which either have positive real part, or which have
  zero real part but which have non-negative imaginary part.
\end{enumerate}

In a companion paper (\cite{gutin-pinv}), we have made progress on problem 1 by showing that a
sufficient condition for the Jordan SVD to exist is that
\(\operatorname{rank}(A) = \operatorname{rank}(B) = \operatorname{rank}(AB) = \operatorname{rank}(BA)\).
Still, there are matrix pairs \([A,B]\) which have Jordan SVDs but which
don't satisfy this condition. 
The paper \cite{kaplansky} may be relevant, and taking inspiration from
it one might conjecture that a necessary and sufficient condition is for
$AB$ to be similar to $BA$.

In this paper, we made progress on
problem 2 by showing that the Jordan SVD is unique whenever \(A\) and
\(B\) are invertible.

\subsection{Insight into numerical linear algebra over real and complex matrices?}

The main insight of this paper is that some matrix decompositions which appear to be special cases of other matrix decompositions are actually equivalent to them. An example is SVD, which is equivalent to eigendecomposition. This equivalence is revealed only through the use of the double or double-complex numbers. This insight is made more interesting when we observe that this means that a numerical algorithm for computing the SVD can successfully find the eigendecomposition of a matrix simply by generalising the algorithm verbatim to double-complex numbers. In practice though, the resulting eigendecomposition algorithm might fail to converge for some matrices, while the original SVD algorithm might converge for every matrix. The fact that it sometimes works is still intriguing.

But can it ever be useful? The decompositions in section \ref{simple-matrix-decompositions-over-the-double-numbers}, which serve to demonstrate this phenomenon, have some drawbacks. For some matrices, they don't exist, and when they do exist, they may not be computable in a numerically stable manner. While a matrix decomposition over the complex numbers may be stable, its generalisation to the double or double-complex numbers may not be stable (for instance, see the QR decomposition).

Sometimes, a matrix decomposition can be made stable by introducing an additional factor which is a permutation matrix. For instance, compare the LU decomposition $M = LU$ with the PLU decomposition $M = PLU$. It turns out that the naive generalisation of the PLU decomposition to double matrices doesn't exist for all double matrices. We propose therefore to make $P$ an arbitrary \emph{pair} of permutation matrices $[P,Q]$ (understood as a single double matrix). We find that this fixes many existential issues, and not just with the PLU decomposition (which is a fairly trivial example). For instance, this can be done with some stable variants of the LDL decomposition (like the Bunch-Kaufmann-Parlett decomposition) with the result being a stable variant of LU. Unfortunately, it also results in a situation where algorithms for computing matrix decompositions don't generalise in the right way. Algorithms for computing matrix decompositions like the PLU assume that the set of permutation matrices is the usual one. In this paper, we've tried to ensure that the generalisation of an algorithm to double matrices would be (nearly) verbatim, but this may no longer be possible once permutation matrices get involved.

In studying the Jordan SVD, we also studied the polar decomposition (which it is equivalent to). From an existential viewpoint, the polar decomposition is fairly pleasant, because it exists for all invertible double-complex matrices. Can the polar decomposition for double-complex matrices be computed in a numerically stable manner? Also, we're not sure how relevant the double-complex polar decomposition is for people working in practical applications, but we remain hopeful.

Returning to what's been established, we still think the qualitative insight that SVD and eigendecomposition are more tightly linked than previously thought could become useful eventually. The other examples in section \ref{simple-matrix-decompositions-over-the-double-numbers} show tighter links between different decompositions than what one would expect.

\bigskip{\bf Acknowledgments.} I would like to thank Gregory Gutin for many useful comments and suggestions.

\begin{figure}[h]
  \caption{LR algorithm for computing SVD over double-complex matrices. (Note that Sympy's inbuilt LDL decomposition only works for positive-definite matrices, so we needed to make our own version from scratch).}
  \label{splitcomplex-svdcode}
  \begin{Shaded}
  \begin{Highlighting}[]
  \ImportTok{from}\NormalTok{ sympy }\ImportTok{import} \OperatorTok{*}
  
  \NormalTok{J }\OperatorTok{=}\NormalTok{ var(}\StringTok{\textquotesingle{}J\textquotesingle{}}\NormalTok{)}
  
  \KeywordTok{def}\NormalTok{ simplify\_split(e):}
  \NormalTok{    real\_part }\OperatorTok{=}\NormalTok{ (e.subs(J,}\DecValTok{1}\NormalTok{) }\OperatorTok{+}\NormalTok{ e.subs(J,}\OperatorTok{{-}}\DecValTok{1}\NormalTok{))}\OperatorTok{/}\DecValTok{2}
  \NormalTok{    imag\_part }\OperatorTok{=}\NormalTok{ (e.subs(J,}\DecValTok{1}\NormalTok{) }\OperatorTok{{-}}\NormalTok{ e.subs(J,}\OperatorTok{{-}}\DecValTok{1}\NormalTok{))}\OperatorTok{/}\DecValTok{2}
      \ControlFlowTok{return}\NormalTok{ simplify(real\_part }\OperatorTok{+}\NormalTok{ J}\OperatorTok{*}\NormalTok{imag\_part)    }
  
  \KeywordTok{def}\NormalTok{ LDL(A):}
      \CommentTok{"""Implements LDL decomposition for some (most?) Hermitian matrices,}
  \CommentTok{    including some indefinite ones."""}
  \NormalTok{    n }\OperatorTok{=}\NormalTok{ A.rows}
  \NormalTok{    D }\OperatorTok{=}\NormalTok{ zeros(n,n)}
  \NormalTok{    L }\OperatorTok{=}\NormalTok{ eye(n)}
      \ControlFlowTok{for}\NormalTok{ i }\KeywordTok{in} \BuiltInTok{range}\NormalTok{(n):}
          \ControlFlowTok{for}\NormalTok{ j }\KeywordTok{in} \BuiltInTok{range}\NormalTok{(n):}
              \ControlFlowTok{if}\NormalTok{ i }\OperatorTok{==}\NormalTok{ j:}
  \NormalTok{                D[j,j] }\OperatorTok{=}\NormalTok{ A[j,j] }\OperatorTok{{-}} \BuiltInTok{sum}\NormalTok{(L[j,k] }\OperatorTok{*}\NormalTok{ L[j,k].subs(J,}\OperatorTok{{-}}\NormalTok{J) }\OperatorTok{*}\NormalTok{ D[k,k] }
                                        \ControlFlowTok{for}\NormalTok{ k }\KeywordTok{in} \BuiltInTok{range}\NormalTok{(}\DecValTok{0}\NormalTok{, j))}
  \NormalTok{                D[j,j] }\OperatorTok{=}\NormalTok{ simplify\_split(D[j,j])}
              \ControlFlowTok{if}\NormalTok{ i }\OperatorTok{\textgreater{}}\NormalTok{ j:            }
  \NormalTok{                L[i,j] }\OperatorTok{=} \DecValTok{1}\OperatorTok{/}\NormalTok{D[j,j] }\OperatorTok{*}\NormalTok{ (}
  \NormalTok{                    A[i,j] }\OperatorTok{{-}} \BuiltInTok{sum}\NormalTok{(L[i,k] }\OperatorTok{*}\NormalTok{ L[j,k].subs(J,}\OperatorTok{{-}}\NormalTok{J) }\OperatorTok{*}\NormalTok{ D[k,k]}
                                   \ControlFlowTok{for}\NormalTok{ k }\KeywordTok{in} \BuiltInTok{range}\NormalTok{(}\DecValTok{0}\NormalTok{, j)))}
  \NormalTok{                L[i,j] }\OperatorTok{=}\NormalTok{ simplify\_split(L[i,j])}
      \ControlFlowTok{return}\NormalTok{ L, D}
  
  
  \KeywordTok{def}\NormalTok{ svd(A, iters}\OperatorTok{=}\DecValTok{20}\NormalTok{):}
  \NormalTok{    M }\OperatorTok{=}\NormalTok{ A.T.subs(J,}\OperatorTok{{-}}\NormalTok{J) }\OperatorTok{*}\NormalTok{ A}
      \ControlFlowTok{for}\NormalTok{ \_ }\KeywordTok{in} \BuiltInTok{range}\NormalTok{(iters):}
  \NormalTok{        L, D }\OperatorTok{=}\NormalTok{ LDL(M)}
  \NormalTok{        M }\OperatorTok{=}\NormalTok{ D }\OperatorTok{**}\NormalTok{ Rational(}\DecValTok{1}\NormalTok{,}\DecValTok{2}\NormalTok{) }\OperatorTok{*}\NormalTok{ L.T.subs(J,}\OperatorTok{{-}}\NormalTok{J) }\OperatorTok{*}\NormalTok{ L }\OperatorTok{*}\NormalTok{ D }\OperatorTok{**}\NormalTok{ Rational(}\DecValTok{1}\NormalTok{,}\DecValTok{2}\NormalTok{)}
      \ControlFlowTok{return}\NormalTok{ D }\OperatorTok{**}\NormalTok{ Rational(}\DecValTok{1}\NormalTok{,}\DecValTok{2}\NormalTok{)}
  \end{Highlighting}
  \end{Shaded}
\end{figure}
  
\begin{figure}[h]
  \caption{Gram-Schmidt implementation for double matrices.}
  \label{gramschmidt-code}
  \begin{Shaded}
  \begin{Highlighting}[]
  \ImportTok{from}\NormalTok{ sympy }\ImportTok{import} \OperatorTok{*}
  
  \NormalTok{J }\OperatorTok{=}\NormalTok{ var(}\StringTok{\textquotesingle{}J\textquotesingle{}}\NormalTok{)}
  
  \KeywordTok{def}\NormalTok{ simplify\_split(e):}
  \NormalTok{    real\_part }\OperatorTok{=}\NormalTok{ (e.subs(J,}\DecValTok{1}\NormalTok{) }\OperatorTok{+}\NormalTok{ e.subs(J,}\OperatorTok{{-}}\DecValTok{1}\NormalTok{))}\OperatorTok{/}\DecValTok{2}
  \NormalTok{    imag\_part }\OperatorTok{=}\NormalTok{ (e.subs(J,}\DecValTok{1}\NormalTok{) }\OperatorTok{{-}}\NormalTok{ e.subs(J,}\OperatorTok{{-}}\DecValTok{1}\NormalTok{))}\OperatorTok{/}\DecValTok{2}
      \ControlFlowTok{return}\NormalTok{ simplify(real\_part }\OperatorTok{+}\NormalTok{ J}\OperatorTok{*}\NormalTok{imag\_part)    }
  
  \KeywordTok{def}\NormalTok{ gram\_schmidt\_qr(m):}
  \NormalTok{    u }\OperatorTok{=}\NormalTok{ zeros(m.rows,m.cols)}
      \ControlFlowTok{for}\NormalTok{ k }\KeywordTok{in} \BuiltInTok{range}\NormalTok{(m.cols):}
  \NormalTok{        u[:,k] }\OperatorTok{=}\NormalTok{ m[:,k] }\OperatorTok{{-}} \BuiltInTok{sum}\NormalTok{((proj(u[:,j],m[:,k]) }\ControlFlowTok{for}\NormalTok{ j }\KeywordTok{in} \BuiltInTok{range}\NormalTok{(k)),}
  \NormalTok{                              start}\OperatorTok{=}\NormalTok{zeros(m.rows, }\DecValTok{1}\NormalTok{))}
  \NormalTok{    Q }\OperatorTok{=}\NormalTok{ simplify\_split(u }\OperatorTok{*}\NormalTok{ split\_dot(u,u) }\OperatorTok{**}\NormalTok{ Rational(}\OperatorTok{{-}}\DecValTok{1}\NormalTok{,}\DecValTok{2}\NormalTok{))}
  \NormalTok{    R }\OperatorTok{=}\NormalTok{ zeros(m.rows, m.cols)}
      \ControlFlowTok{for}\NormalTok{ i }\KeywordTok{in} \BuiltInTok{range}\NormalTok{(m.cols):}
          \ControlFlowTok{for}\NormalTok{ j }\KeywordTok{in} \BuiltInTok{range}\NormalTok{(m.cols):}
              \ControlFlowTok{if}\NormalTok{ i }\OperatorTok{\textless{}=}\NormalTok{ j:}
  \NormalTok{                R[i,j] }\OperatorTok{=}\NormalTok{ split\_dot(Q[:,i], m[:,j])}
  \NormalTok{    R }\OperatorTok{=}\NormalTok{ simplify\_split(R)}
      \ControlFlowTok{return}\NormalTok{ Q, R}
  
  \KeywordTok{def}\NormalTok{ split\_dot(u, v):}
      \ControlFlowTok{return}\NormalTok{ u.T.subs(J,}\OperatorTok{{-}}\NormalTok{J) }\OperatorTok{*}\NormalTok{ v}
  
  \KeywordTok{def}\NormalTok{ proj(u, a):}
      \ControlFlowTok{return}\NormalTok{ u }\OperatorTok{*}\NormalTok{ split\_dot(u,a) }\OperatorTok{*}\NormalTok{ split\_dot(u,u).inv()}
  \end{Highlighting}
  \end{Shaded}
\end{figure}

\begin{figure}[h]
  \caption{Householder QR implementation for double matrices.}
  \label{householdercode}
\begin{Shaded}
\begin{Highlighting}
\KeywordTok{def}\NormalTok{ householder\_reflection(x):}
    \CommentTok{"""Computes a householder reflection from a split{-}complex vector."""}
    \ControlFlowTok{for}\NormalTok{ i }\KeywordTok{in} \BuiltInTok{range}\NormalTok{(x.rows):}
        \ControlFlowTok{if}\NormalTok{ scabs(x[i]) }\OperatorTok{!=} \DecValTok{0}\NormalTok{:}
            \ControlFlowTok{break}
        \ControlFlowTok{if}\NormalTok{ i }\OperatorTok{==}\NormalTok{ x.rows }\OperatorTok{{-}} \DecValTok{1}\NormalTok{:}
            \ControlFlowTok{return}\NormalTok{ eye(x.rows)}
    \ControlFlowTok{if}\NormalTok{ i }\OperatorTok{!=} \DecValTok{0}\NormalTok{:}
\NormalTok{        initial\_rotation }\OperatorTok{=}\NormalTok{ eye(x.rows)}
\NormalTok{        initial\_rotation[}\DecValTok{0}\NormalTok{,}\DecValTok{0}\NormalTok{] }\OperatorTok{=}\NormalTok{ initial\_rotation[i,i] }\OperatorTok{=} \DecValTok{0}
\NormalTok{        initial\_rotation[i,}\DecValTok{0}\NormalTok{] }\OperatorTok{=} \OperatorTok{{-}}\DecValTok{1}
\NormalTok{        initial\_rotation[}\DecValTok{0}\NormalTok{,i] }\OperatorTok{=} \DecValTok{1}
        \ControlFlowTok{return}\NormalTok{ householder\_reflection(initial\_rotation }\OperatorTok{*}\NormalTok{ x) }\OperatorTok{*}\NormalTok{ initial\_rotation}
\NormalTok{    alpha }\OperatorTok{=}\NormalTok{ simplify\_split(}\OperatorTok{{-}}\NormalTok{x[}\DecValTok{0}\NormalTok{] }\OperatorTok{/}\NormalTok{ vec\_norm(Matrix([x[}\DecValTok{0}\NormalTok{]])) }\OperatorTok{*}\NormalTok{ vec\_norm(x))}
\NormalTok{    e1 }\OperatorTok{=}\NormalTok{ simplify\_split(Matrix([}\DecValTok{1}\NormalTok{] }\OperatorTok{+}\NormalTok{ [}\DecValTok{0}\NormalTok{] }\OperatorTok{*}\NormalTok{ (x.rows }\OperatorTok{{-}} \DecValTok{1}\NormalTok{)))}
\NormalTok{    u }\OperatorTok{=}\NormalTok{ simplify\_split(x }\OperatorTok{{-}}\NormalTok{ e1 }\OperatorTok{*}\NormalTok{ alpha)}
\NormalTok{    v }\OperatorTok{=}\NormalTok{ simplify\_split(u }\OperatorTok{/}\NormalTok{ vec\_norm(u))}
    \ControlFlowTok{return}\NormalTok{ simplify\_split(expand(eye(x.rows) }\OperatorTok{{-}} \DecValTok{2} \OperatorTok{*}\NormalTok{ v }\OperatorTok{*}\NormalTok{ v.T.subs(J,}\OperatorTok{{-}}\NormalTok{J)))}

\KeywordTok{def}\NormalTok{ householderQR(m):}
    \CommentTok{"""Computes the Q matrix in the QR decomposition using Householder}
\CommentTok{    reflections"""}
    \ControlFlowTok{if}\NormalTok{ m.cols }\OperatorTok{==} \DecValTok{1}\NormalTok{:}
        \ControlFlowTok{return}\NormalTok{ householder\_reflection(m)}
\NormalTok{    Q1 }\OperatorTok{=}\NormalTok{ simplify\_split(householder\_reflection(m[:,}\DecValTok{0}\NormalTok{]))}
\NormalTok{    Qrest }\OperatorTok{=}\NormalTok{ householderQR(simplify\_split((Q1 }\OperatorTok{*}\NormalTok{ m)[}\DecValTok{1}\NormalTok{:,}\DecValTok{1}\NormalTok{:]))}
    \ControlFlowTok{return}\NormalTok{ simplify\_split(Q1.T.subs(J,}\OperatorTok{{-}}\NormalTok{J) }\OperatorTok{*}\NormalTok{ diag(eye(}\DecValTok{1}\NormalTok{),Qrest))}
\end{Highlighting}
\end{Shaded}
\end{figure}

\bibliographystyle{unsrt}
\bibliography{paper1_1}

\begin{thebibliography}{10}

\bibitem{doublecomplex1}
G.~Baley Price.
\newblock {\em An Introduction to Multicomplex Spaces and Functions}.
\newblock Chapman \& Hall/CRC Pure and Applied Mathematics. Taylor \& Francis,
  1990.

\bibitem{sympy}
Aaron Meurer, Christopher~P. Smith, Mateusz Paprocki, Ond\v{r}ej
  \v{C}ert\'{i}k, Sergey~B. Kirpichev, Matthew Rocklin, AMiT Kumar, Sergiu
  Ivanov, Jason~K. Moore, Sartaj Singh, Thilina Rathnayake, Sean Vig, Brian~E.
  Granger, Richard~P. Muller, Francesco Bonazzi, Harsh Gupta, Shivam Vats,
  Fredrik Johansson, Fabian Pedregosa, Matthew~J. Curry, Andy~R. Terrel,
  \v{S}t\v{e}p\'{a}n Rou\v{c}ka, Ashutosh Saboo, Isuru Fernando, Sumith Kulal,
  Robert Cimrman, and Anthony Scopatz.
\newblock Sympy: symbolic computing in python.
\newblock {\em PeerJ Computer Science}, 3:e103, January 2017.

\bibitem{eigen}
Ga\"{e}l Guennebaud, Beno\^{i}t Jacob, et~al.
\newblock Eigen v3.
\newblock http://eigen.tuxfamily.org, 2010.

\bibitem{sagemath}
{The Sage Developers}.
\newblock {\em {S}ageMath, the {S}age {M}athematics {S}oftware {S}ystem
  ({V}ersion 9.2)}, 2020.
\newblock {\tt https://www.sagemath.org}.

\bibitem{cholesky}
{Nicholas J.} Higham.
\newblock Cholesky factorization.
\newblock {\em Wiley Interdisciplinary Reviews: Computational Statistics},
  1(2):251--254, September 2009.

\bibitem{computations}
Gene~H. Golub and Charles~F. van Loan.
\newblock {\em Matrix Computations}.
\newblock JHU Press, fourth edition, 2013.

\bibitem{strang09}
Gilbert Strang.
\newblock {\em Introduction to Linear Algebra}.
\newblock Wellesley-Cambridge Press, Wellesley, MA, fourth edition, 2009.

\bibitem{implicitcholesky}
K.~Vince~Fernando and Beresford~N. Parlett.
\newblock Implicit {C}holesky algorithms for singular values and vectors of
  triangular matrices.
\newblock {\em Numerical Linear Algebra with Applications}, 2(6):507--531,
  1995.

\bibitem{lrproof}
Heinz Rutishauser.
\newblock Solution of eigenvalue problems with the lr-transformation.
\newblock {\em National Bureau of Standards, Applied Mathematics Series},
  49:47--81, 1958.

\bibitem{algebraicpolar}
Dipa Choudhury and Roger~A. Horn.
\newblock A complex orthogonal-symmetric analog of the polar decomposition.
\newblock {\em SIAM Journal on Algebraic Discrete Methods}, 8(2):219--225,
  1987.

\bibitem{kaplansky}
Irving Kaplansky.
\newblock Algebraic polar decomposition.
\newblock {\em SIAM J. Matrix Anal. Appl.}, 11(2):213–217, March 1990.

\bibitem{bunchparlett}
J.~R. Bunch and B.~N. Parlett.
\newblock Direct methods for solving symmetric indefinite systems of linear
  equations.
\newblock {\em SIAM Journal on Numerical Analysis}, 8(4):639--655, 1971.

\bibitem{bunchkaufman}
James~R Bunch and Linda Kaufman.
\newblock Some stable methods for calculating inertia and solving symmetric
  linear systems.
\newblock {\em Mathematics of computation}, pages 163--179, 1977.

\bibitem{rrqr}
Peter Businger and Gene~H. Golub.
\newblock Linear least squares solutions by householder transformations.
\newblock {\em Numer. Math.}, 7(3):269–276, June 1965.

\bibitem{gutin-pinv}
Ran Gutin.
\newblock An analogue of the relationship between svd and pseudoinverse over
  double-complex matrices, 2021.

\end{thebibliography}

\end{document}